\documentclass[11pt]{article}

\usepackage[margin=1.08in]{geometry}
\usepackage[T1]{fontenc}
\usepackage{lmodern}
\usepackage{microtype}
\usepackage{amsmath,amssymb,amsthm,mathtools}
\usepackage{aliascnt}
\usepackage{enumitem}
\usepackage{booktabs}
\usepackage{authblk}
\usepackage{xcolor}
\usepackage[hidelinks]{hyperref}
\usepackage[nameinlink,noabbrev]{cleveref}
\crefname{theorem}{theorem}{theorems}
\crefname{lemma}{lemma}{lemmas}
\crefname{proposition}{proposition}{propositions}
\crefname{corollary}{corollary}{corollaries}

\newtheorem{theorem}{Theorem}[section]
\newaliascnt{proposition}{theorem}
\newtheorem{proposition}[proposition]{Proposition}
\aliascntresetthe{proposition}
\newaliascnt{lemma}{theorem}
\newtheorem{lemma}[lemma]{Lemma}
\aliascntresetthe{lemma}
\newaliascnt{corollary}{theorem}
\newtheorem{corollary}[corollary]{Corollary}
\aliascntresetthe{corollary}
\theoremstyle{definition}
\newaliascnt{definition}{theorem}

\aliascntresetthe{definition}
\newaliascnt{remark}{theorem}

\aliascntresetthe{remark}

\newcommand{\bits}{\{0,1\}^{*}}
\newcommand{\eps}{\varepsilon}
\newcommand{\Sub}{\operatorname{Sub}}
\newcommand{\zz}{\operatorname{z}}
\newcommand{\oo}{\operatorname{o}}
\newcommand{\wt}{\operatorname{wt}}
\newcommand{\lexlt}{<_{\mathrm{lex}}}

\newcommand{\PhiNeg}{\Phi_{-}}
\newcommand{\col}[2]{\begin{pmatrix}#1\\#2\end{pmatrix}}
\newcommand{\email}[1]{\href{mailto:#1}{\nolinkurl{#1}}}

\title{A Complete Proof for Tu-Deng Conjecture}
\author[1]{Renzhang Liu\thanks{\email{liurenzhang@amss.ac.cn}}}
\author[1]{Hengyi Luo\thanks{\email{luohengyi23@mails.ucas.ac.cn}}}
\author[1]{Tianyuan Xie\thanks{\email{terencexty@gmail.com}}}
\affil[1]{State Key Laboratory of Mathematical Sciences,
Academy of Mathematics and Systems Science, Chinese Academy of Sciences,
Beijing, China}
\date{30, July 2026}

\begin{document}
\maketitle

\begin{abstract}
Let $N=2^k-1$ and let $\wt(n)$ denote the binary Hamming weight.  The
Tu-Deng conjecture asserts that, for every $1\le t\le N-1$, at most
$2^{k-1}$ pairs $(a,b)\in\{0,\ldots,N-1\}^2$ satisfy
$a+b\equiv t\pmod N$ and $\wt(a)+\wt(b)<k$. Partial results are known.
We give a complete proof of this conjecture. We first show that the Tu-Deng
counts equals the number of cyclic carry solutions for which $\wt(B)-\wt(A)<0$
and $A+t\equiv B\pmod N$. The enumerator of the cyclic carry solutions factors as
\[
  C_v = 1+(X+Y-1)J_v+X^{\zz(v)+1}Y^{\oo(v)+1},
\]
where $t=10v$ is the binary expansion of $t$(least significant bits first) and $J_v$ enumerates the language
\[
  \Sub(v)\ \mathbin{\dot\cup}\
  \{u\in\partial_1\Sub(v):u<_{\rm lex}v\}.
\]
Estimating the strict negative half-plane mass of $C_v$ gives the desired bound.
\end{abstract}

\section{Introduction}

For a nonnegative integer $n$, write $\wt(n)$ for the number of $1$'s in
its binary expansion.  Tu and Deng introduced the following conjecture in
connection with Boolean functions of optimal algebraic immunity
\cite{TuDeng2011}.

\begin{theorem}[Tu-Deng conjecture]\label{thm:main}
Let $k\ge2$, put $N=2^k-1$, and let $1\le t\le N-1$.  Then
\[
 \left|\left\{(a,b)\in\{0,\ldots,N-1\}^2:
 a+b\equiv t\pmod N,
 \ \wt(a)+\wt(b)<k\right\}\right|
 \le 2^{k-1}.
\]
\end{theorem}

Tu and Deng verified the conjecture for $k\leq 29$ \cite{TuDeng2011}. Later Flori verified the conjecture for $k$ up to $40$ \cite{Flori2012}. This conjecture proves to be true for $t$ with some special form \cite{CGGC2019, CLS2011, ChenLinWei2020, CHZ2015, DY2012, FRCM2010, QSF2016} .

Spiegelhofer and Wallner proved the
conjecture for a set of parameters of asymptotic density one
\cite{SpiegelhoferWallner2019}.  A recent proof of Cusick's related
sum-of-digits conjecture uses first-exit laws of principal subsequence ideals
and marked deletion counts \cite{Cheng2026}.

We give a complete proof in the work. Our proof consists of the following steps.
Firstly, by a complementing
change of variables and enumerating the cyclic carries, we reduce Tu-Deng count to computing the strict
negative half-plane $(\deg(X)>\deg(Y))$ mass of
\[
  C_v(X,Y)=B_{0v}^1(X,Y)+B_v^0(X,Y),
\]
where $B_w^b$ enumerates the directional boundary of the language $D(w)$.  Secondly,
we show that $C_v$ can be factorized as
\[
  C_v=1+(X+Y-1)(H_v+I_v)+X^{\zz(v)+1}Y^{\oo(v)+1}.
\]
where $H_v$ is modelled as the enumerator of a language equipped with a lexicographical order.
Finally, we show that the strict
negative half-plane functional annihilates
all coefficient layers of $C_v$ except the diagonal and the immediately
lower diagonal. Comparing those two layers term by term yields the desired inequality.

\section{Directional boundaries}

Let $\bits$ be the set of finite binary words, including the empty word
$\eps$.  For $u\in\bits$, let $|u|$ be its length, and let
\[
  \zz(u)=|u|-\wt(u),
  \qquad
  \oo(u)=\wt(u)
\]
be the number of $0$'s and $1$'s in $u$ respectively.
Write $u\preceq v$ when $u$ is a subsequence of
$v$, and put
\[
  D(v)=\Sub(v)=\{u\in\bits:u\preceq v\}.
\]
The set $D(v)$ is finite and closed under deletion, hence in particular
prefix-closed.\\
For a finite prefix-closed language $D\subset\bits$ and $b\in\{0,1\}$,
define its directional boundary by
\[
  \partial_bD=\{ub:u\in D,\ ub\notin D\}.
\]
It follows that $\partial_0D(v), \partial_1D(v), D(v)$ are pair-wise disjoint.
That is,
$$\partial_0D(v)\cap \partial_1D(v)=\partial_bD(v)\cap D(v)=\emptyset.$$
For indeterminates $X,Y$, set
\[
  \omega(u)=X^{\zz(u)}Y^{\oo(u)}.
\]
For a generator $v$, define
\[
  I_v=\sum_{u\in D(v)}\omega(u),
  \qquad
  B_v^b=\sum_{u\in\partial_bD(v)}\omega(u),
  \qquad
  B_v=\col{B_v^0}{B_v^1}.
\]

\begin{lemma}\label{lem:boundary-recursion}
Let
\[
 A_0=\begin{pmatrix}X&0\\Y&1\end{pmatrix},
 \qquad
 A_1=\begin{pmatrix}1&X\\0&Y\end{pmatrix}.
\]
Then
\[
  B_{v0}=A_0B_v,
  \qquad
  B_{v1}=A_1B_v,
  \qquad
  B_\eps=\col{X}{Y}.
\]
\end{lemma}

\begin{proof}
We show the following boundary
decompositions holds.
\begin{align*}
 \partial_0D(v0)
   &=\{x0:x\in\partial_0D(v)\},\\
 \partial_1D(v0)
   &=\partial_1D(v)\mathbin{\dot\cup}
     \{x1:x\in\partial_0D(v)\}.
\end{align*}

The first equality follows easily by appending a $0$ to words in $\partial_0D(v)$ and deleting the final $0$ from  words in $\partial_0D(v0)$. We only prove the second one.

By definition, $\partial_1D(v0)=\{x1:x\in D(v0),\ x1\notin D(v0)\}.$

Since $\partial_1D(v)\subseteq\partial_1D(v0)$, it follows easily that
$$\partial_1D(v0)\supseteq \partial_1D(v){\cup} \{x1:x\in\partial_0D(v)\}.$$
Let $w=x1\in\partial_1D(v0)$, where $x\in D(v0)$.
\begin{enumerate}
\item If $x\in D(v),$ then $w=x1\notin D(v)$. Otherwise, $w=x1\in D(v)\subseteq D(v0)$, contrary to the definition of $\partial_1D(v0)$. Therefore, $w\in \partial_1D(v)$.
\item If $x\notin D(v),$ then $x=u0\in D(v0)\setminus D(v)$. Hence $u\in D(v)$ and $u0\notin D(v)$ and $x\in \partial_0D(v)$ by definition. Therefore, $w\in \{x1:x\in\partial_0D(v)\}.$
\end{enumerate}
Then $$\partial_1D(v0)\subseteq \partial_1D(v){\cup} \{x1:x\in\partial_0D(v)\}.$$
Since $D(v)\cap \partial_0D(v)=\emptyset,$
$$\partial_1D(v){\cap} \{x1:x\in\partial_0D(v)\}=\{x1: x\in D(v), x1\notin D(v), x\in\partial_0D(v)\}=\emptyset.$$ Then
$$\partial_1D(v0)=\partial_1D(v)\mathbin{\dot\cup}\{x1:x\in\partial_0D(v)\}.$$
This proves the second equality and we have
$$B_{v0}^0=XB_v^0,\quad  B_{v0}^1=YB_v^0+B_v^1,\quad  B_{\eps}^0=X,\quad  B_{\eps}^1=Y.$$
Symmetrically,
\begin{align*}
 \partial_0D(v1)
   &=\partial_0D(v)\mathbin{\dot\cup}
     \{x0:x\in\partial_1D(v)\},\\
 \partial_1D(v1)
   &=\{x1:x\in\partial_1D(v)\},
\end{align*}
and $$B_{v1}^0=B_v^0+XB_v^1,\quad  B_{v1}^1=YB_v^1.$$ This concludes the proof.

%For example, if $x=a0\in\partial_0D(v)$, then $x\in D(v0)$ by using the
%new final $0$, whereas $x0\notin D(v0)$; the converse follows by asking
%whether the prefix of a boundary word uses the new final position.  The
%recurrences for $v1$ are symmetric:
%\begin{align*}
% \partial_0D(v1)
%   &=\partial_0D(v)\mathbin{\dot\cup}
%     \{x0:x\in\partial_1D(v)\},\\
% \partial_1D(v1)
%   &=\{x1:x\in\partial_1D(v)\}.
%\end{align*}
%Taking weighted enumerators gives the stated matrices.  For the empty
%generator the two boundary words are $0$ and $1$.
\end{proof}

Put
\[
  \Lambda=X+Y-1.
\]

\begin{lemma}\label{lem:tree-identity}
For every binary word $v$,
\[
  B_v^0+B_v^1=1+\Lambda I_v.
\]
\end{lemma}

\begin{proof}
We generate a finite full rooted binary tree as follows. Label the root as $\eps$.
The weight of a node with label $u$ is defined as $\omega(u)$.
If the label $w$ of a node belongs to $D(v)$, expand the node by generating
a left and a right child, whose labels are $w0$ and $w1$ respectively. Then it
can be seen that the labels of internal nodes and the leaves are exactly $D(v)$ and
$\partial_0D(v)\mathbin{\dot\cup}\partial_1D(v)$ respectively.

When a node of weight $\omega(u)$ is expanded, the weight of leaves is increased by
\[
  X\omega(u)+Y\omega(u)-\omega(u)=\Lambda\omega(u).
\]
Starting from the root with weight $1$ and expanding every node in
$D(v)$ proves the identity.

%{
%\textcolor{red}{Label each node as a word. The label of the left child is obtained by appending a $0$ in the label of the father node and the label of the right child is obtained by appending a $1$. If a label $w\in D(v)$, expand the node by generating its left and right child. The tree is generated by starting with the root node labeled as $\eps$. Then the tree is a full rooted binary tree and the leaves are exactly $\partial_0D(v)\mathbin{\dot\cup}\partial_1D(v)$. }
%}

\end{proof}

We use the lexicographic order on $\bits$ determined by $0<1$, with a proper
prefix smaller than the word it prefixes.

\section{Cyclic addition}

For this section fix $k\ge2$ and $1\le t\le2^k-2$, and put $N=2^k-1$.
Multiplication by $2$ modulo $N$ cyclically rotates the $k$ binary digits and
preserves Hamming weight.  Applying the same rotation to both coordinates of
a pair is a bijection in the Tu-Deng count.  Since the $k$-digit word of
$t$ is neither constant $0$ nor constant $1$, we may therefore choose the
cyclic reading origin so that
\[
  t_1t_2\cdots t_k=10v
\]
for a word $v$ of length $k-2$.  The positions are read from least to most
significant in this rotated representation.

Given an original pair $(x,y)$, replace $x$ by the $k$-bit complement
\[
  A=N-x\in\{1,\ldots,N\}
\]
and put $B=y\in\{0,\ldots,N-1\}$.  Thus
\[
  \wt(x)=k-\wt(A).
\]
Consequently
\[
  x+y\equiv t\pmod N,
  \quad \wt(x)+\wt(y)<k
\]
is equivalent to
\[
  B\equiv A+t\pmod N,
  \quad \wt(B)-\wt(A)<0.
\]

A
\emph{cyclic carry solution} consists of words
$A=(a_1,\ldots,a_k)$ and $B=(b_1,\ldots,b_k)$ together with carries
$c_0,\ldots,c_k\in\{0,1\}$ such that $c_k=c_0$ and
\begin{equation}\label{eq:cyclic-carry}
  a_i+t_i+c_{i-1}=b_i+2c_i
  \qquad(1\le i\le k).
\end{equation}
Summing \eqref{eq:cyclic-carry} with powers of $2$ gives
\begin{equation}\label{eq:numeric-carry}
  B=A+t-c_0N.
\end{equation}
Conversely, if \eqref{eq:numeric-carry} holds, then
$A+t+c_0=B+c_0 2^k$; ordinary binary addition with initial carry $c_0$
therefore has output $B$ and final carry $c_0$, and constructs the required
cyclic carry solution.

\begin{lemma}\label{lem:carry-uniqueness}
Fix $k$-bit words $A,B$ and suppose that $B\equiv A+t\pmod N$.  There is at
most one cyclic carry solution with word components $A,B$.  More precisely,
the initial carry $c_0$ and all subsequent carries $c_1,\ldots,c_k$ are
uniquely determined by $A,B$.
\end{lemma}

\begin{proof}
Since $0\le A,B\le N$ and $1\le t\le N-1$,
\[
  -N<A+t-B<2N.
\]
The integer $A+t-B$ is also divisible by $N$, and hence
\[
  A+t-B\in\{0,N\}.
\]
Equation \eqref{eq:numeric-carry} therefore uniquely determines
\[
  c_0=\frac{A+t-B}{N}\in\{0,1\}.
\]
Once $c_0$ is fixed, ordinary binary addition from least to most significant
bit is deterministic: at each step, $a_i,t_i,c_{i-1}$ uniquely determine the
output bit $b_i$ and the next carry $c_i$.  Thus all carries are unique.
\end{proof}

\begin{lemma}\label{lem:endpoints}
The original Tu-Deng count equals the number of cyclic carry solutions for
which $\wt(B)-\wt(A)<0$.
\end{lemma}

\begin{proof}
The complementing map above sends each original pair to a cyclic carry
solution with $A\ne0^k$ and $B\ne1^k$, and it preserves the strict weight
condition in the displayed form.  Conversely, every cyclic solution with
$A\ne0^k$ and $B\ne1^k$ gives an original pair by $x=N-A$ and $y=B$.

It remains only to check that the two excluded endpoint representations can
never have negative weight difference.  If $A=0^k$, then
\eqref{eq:numeric-carry} forces $c_0=0$ and $B=t$, so
$\wt(B)-\wt(A)=\wt(t)>0$.  If $B=1^k$, then $B=N$ and
\eqref{eq:numeric-carry} forces $c_0=0$ and $A=N-t$; hence
\[
  \wt(B)-\wt(A)=k-(k-\wt(t))=\wt(t)>0.
\]
Thus restricting to the negative side automatically removes exactly the
nonstandard endpoint solutions.  By \cref{lem:carry-uniqueness}, each pair
$(A,B)$ corresponds to at most one cyclic carry solution, so no additional
multiplicity is hidden in this count.
\end{proof}

\begin{lemma}\label{lem:greedy-subsequence}
Let $q$ and $u=u_1\cdots u_m$ be finite words.  Scan the letters of $u$ in
order, matching each $u_j$ to the earliest equal letter of $q$ strictly after
the preceding matched position.  This greedy procedure matches all of $u$ if
and only if $u\preceq q$.
\end{lemma}

\begin{proof}
If the greedy procedure succeeds, its chosen positions are a subsequence
embedding.  Conversely, suppose that an embedding exists at positions
\[
  1\le e_1<\cdots<e_m\le |q|.
\]
If $g_1,g_2,\ldots$ are the greedy positions, then induction gives
$g_j\le e_j$.  The assertion is clear for $j=1$.  If
$g_{j-1}\le e_{j-1}$, then $e_j>e_{j-1}\ge g_{j-1}$ is feasible at step
$j$, so the earliest feasible position satisfies $g_j\le e_j$.  Hence the
greedy procedure cannot fail whenever an embedding exists.
\end{proof}

Call a position $i$ \emph{active} when $c_{i-1}\ne t_i$.  At an inactive
position one has $b_i=a_i$ and $c_i=c_{i-1}$.  At an active position define
\[
  \xi=1-a_i\in\{0,1\}.
\]
A direct check of \eqref{eq:cyclic-carry} gives
\begin{equation}\label{eq:active-table}
  b_i-a_i=2\xi-1,
  \qquad
  c_i=1-\xi.
\end{equation}
After this position, the carry remains $1-\xi$ until the first later digit of
$t$ equal to $\xi$; that position is the next active one.  Thus the successive
active choices execute exactly the earliest-position greedy matching from
\cref{lem:greedy-subsequence} on the unscanned suffix, and the last choice is
the first symbol that cannot be matched.

For a polynomial $F=\sum_{a,b}f_{a,b}X^aY^b$, possibly with signed
coefficients, define the strict negative half-plane functional
\begin{equation}\label{eq:phi-def}
  \PhiNeg(F)=\sum_{a>b} f_{a,b}\,2^{-a-b}.
\end{equation}
Finally put
\begin{equation}\label{eq:C-def}
  C_v=B_{0v}^1+B_v^0.
\end{equation}

\begin{proposition}\label{prop:first-exit}
Let $S_{t,k}$ denote the set in \cref{thm:main}.  With the above rotation
$t=10v$,
\[
  \frac{|S_{t,k}|}{2^k}=\PhiNeg(C_v).
\]
\end{proposition}

\begin{proof}
First suppose $c_0=0$.  Since $t_1=1$, the first position is active.  Let
$\xi_1\cdots\xi_m$ be the successive active choices.  By the observation
after \eqref{eq:active-table} and \cref{lem:greedy-subsequence}, the proper
prefix $\xi_1\cdots\xi_{m-1}$ is a subsequence of the remaining word $0v$,
while the full word is not.  Cyclic closure gives
$c_k=1-\xi_m=c_0=0$, so $\xi_m=1$.  Hence the active-choice word belongs to
$\partial_1D(0v)$.

If $c_0=1$, the first digit $t_1=1$ is inactive and the second digit
$t_2=0$ is active.  The same argument, now with remaining word $v$, shows
that the active-choice word belongs to $\partial_0D(v)$, because cyclic
closure requires $\xi_m=0$.

Conversely, by \cref{lem:greedy-subsequence}, the proper prefix of each word
in the indicated directional boundary is matched successfully by the greedy
procedure, whereas its final letter necessarily fails.  The earliest matched
positions are unique, and hence uniquely determine all active positions and
satisfy the required cyclic closure.  At its
$m$ active positions the bits of $A$ are fixed by $a_i=1-\xi_i$; every one
of the other $k-m$ inactive bits of $A$ is arbitrary.  Therefore a boundary
word of length $m$ represents exactly $2^{k-m}$ cyclic carry solutions.
Moreover, by \eqref{eq:active-table},
\[
  \wt(B)-\wt(A)
  =\sum_{j=1}^m(2\xi_j-1)
  =\oo(\xi_1\cdots\xi_m)-\zz(\xi_1\cdots\xi_m).
\]
Thus the negative half-plane solutions counted in \cref{lem:endpoints}, divided by
$2^k$, are exactly the terms selected by \eqref{eq:phi-def} from
$B_{0v}^1+B_v^0$.
\end{proof}

\section{The cross-boundary factorization}

We next factor the polynomial $C_v$ from \eqref{eq:C-def}.  Set
\[
  e=\col{-1}{1},
  \qquad
  d_v=X^{\zz(v)+1}Y^{\oo(v)},
  \qquad
  s_v=Yd_v=X^{\zz(v)+1}Y^{\oo(v)+1}.
\]

\begin{lemma}\label{lem:positive-lift}
There are polynomials $U_v,H_v$ with nonnegative integer coefficients such
that
\begin{equation}\label{eq:delta-decomp}
  B_{0v}-B_v
  =s_ve+\Lambda\col{U_v+d_v}{H_v}.
\end{equation}
They are determined by
\[
  (U_\eps,H_\eps,d_\eps)=(0,0,X)
\]
and the recurrences
\begin{align}
 (U_{v0},H_{v0},d_{v0})
   &=(XU_v,\ H_v+YU_v,\ Xd_v),
   \label{eq:lift-zero}\\
 (U_{v1},H_{v1},d_{v1})
   &=(U_v+XH_v+d_v,\ YH_v,\ Yd_v).
   \label{eq:lift-one}
\end{align}
\end{lemma}

\begin{proof}
Write $\delta_v=B_{0v}-B_v$.  Since both words receive the same appended
letter $b$,
\[
  \delta_{vb}=A_b\delta_v.
\]
At the empty word,
\[
 \delta_\eps
 =B_0-B_\eps
 =\col{X^2-X}{XY}
 =XYe+\Lambda\col{X}{0},
\]
which is \eqref{eq:delta-decomp}.  Direct multiplication gives
\[
  A_0e=Xe-\Lambda\col{0}{1},
  \qquad
  A_1e=Ye+\Lambda\col{1}{0}.
\]
Assuming \eqref{eq:delta-decomp} for $v$, apply $A_0$ and use $s_v=Yd_v$:
\[
 \delta_{v0}
 =Xs_ve+\Lambda
   \col{XU_v+Xd_v}{H_v+YU_v}.
\]
This is \eqref{eq:delta-decomp} with \eqref{eq:lift-zero}.  Similarly,
\[
 \delta_{v1}
 =Ys_ve+\Lambda
   \col{U_v+XH_v+d_v+Yd_v}{YH_v},
\]
which is \eqref{eq:delta-decomp} with \eqref{eq:lift-one}.  The displayed
recurrences preserve nonnegative integer coefficients.
\end{proof}

Taking the second coordinate in \eqref{eq:delta-decomp} and using
\cref{lem:tree-identity} gives the desired factorization.

\begin{corollary}\label{cor:cross-factor}
If
\[
  J_v=I_v+H_v,
\]
then
\begin{equation}\label{eq:cross-factor}
  C_v=1+\Lambda J_v+s_v.
\end{equation}
\end{corollary}

\begin{proof}
The second coordinate of \eqref{eq:delta-decomp} is
\[
  B_{0v}^1-B_v^1=s_v+\Lambda H_v.
\]
Therefore
\[
 C_v
 =B_v^0+B_{0v}^1
 =(B_v^0+B_v^1)+s_v+\Lambda H_v
 =1+\Lambda(I_v+H_v)+s_v.
\]
\end{proof}

\section{A lexicographic language model for $H_v$}

The auxiliary polynomial $H_v$ has an explicit word-language
interpretation.  Define
\begin{equation}\label{eq:H-language}
  \mathcal H(v)
  =\{y\in\partial_1D(v):y\lexlt v\}.
\end{equation}

For a word $w$, let $R_w^b$ be the weighted enumerator of those words in
$\partial_bD(w)$ that begin with $0$.  The set decompositions in \cref{lem:boundary-recursion} preserve the first
letter of every boundary word.  Hence the vector
$R_w=(R_w^0,R_w^1)^T$ obeys the same recurrences as $B_w$,

\[
  R_{w0}=A_0R_w,
  \qquad
  R_{w1}=A_1R_w,
  \qquad
  R_\eps=\col{X}{0}.
\]

Let $w=(w_1,w_2,\cdots, w_k)\in\{0,1\}^k$, then
\begin{equation}\label{eq:R_rev}
R_w=A_{w_k}A_{w_{k-1}}\cdots A_{w_1}R_{\eps}.
\end{equation}

%but with initial vector
%\[
%  R_\eps=\col{X}{0}.
%\]
For a position $i$ in a word $v$, write $v_{<i}$ and $v_{>i}$ for the strict
prefix and strict suffix around that position.

\begin{lemma}\label{lem:source-expansion}
The polynomial $H_v$ from \cref{lem:positive-lift} satisfies
\begin{equation}\label{eq:H-source}
  H_v=
  \sum_{i:v_i=1}
  \omega(v_{<i})R_{v_{>i}}^1.
\end{equation}
\end{lemma}

\begin{proof}
Ignoring the source term $d_v$ in \eqref{eq:lift-one}, the pair
$(U_v,H_v)^T$ evolves under the same matrices $A_0,A_1$ as the boundary
vector.  Each occurrence $v_i=1$ injects the source
$(d_{v_{<i}},0)^T$ into the $U$ coordinate.  Since
$(d_{v_{<i}},0)^T=(X\omega(v_{<i}),0)^T=\omega(v_{<i})R_{\eps}$ and $\omega(v_{<i})$ is a monomial,
it follows from \eqref{eq:R_rev} that propagation through the suffix $v_{>i}$
contributes
\[
 \omega(v_{<i})R_{v_{>i}}
\]
to the final pair.  Taking the second coordinate and summing over all source
positions gives \eqref{eq:H-source}.
\end{proof}

\begin{lemma}\label{lem:embedding-position}
Let
\[
  v=a1s,
  \qquad |a|=i-1,
\]
where the displayed $a$ is the literal prefix of $v$ and the letter at
position $i$ is $1$.  If a nonempty word $r$ begins with $0$ and
\[
  ar\preceq v,
\]
then
\[
  r\preceq s.
\]
\end{lemma}

\begin{proof}
Take any embedding of $ar$ in $v$.  The position used for the first letter of
$r$ is at least $|a|+1=i$.  That letter is $0$, whereas $v_i=1$, so the
position is in fact strictly greater than $i$.  All later matched positions
are also greater than $i$, and hence the whole word $r$ embeds in the suffix
$s=v_{>i}$.
\end{proof}

\begin{proposition}\label{prop:H-lex}
\[
  H_v=\sum_{y\in\mathcal H(v)}\omega(y).
\]
Consequently, if
\begin{equation}\label{eq:L-language}
  \mathcal L(v)=D(v)\mathbin{\dot\cup}\mathcal H(v),
\end{equation}
then
\begin{equation}\label{eq:J-enumerator}
  J_v=\sum_{y\in\mathcal L(v)}\omega(y).
\end{equation}
\end{proposition}

\begin{proof}
The right side of \eqref{eq:H-source} enumerates pairs $(i,x)$ satisfying
\[
  v_i=1,
  \qquad
  x\in\partial_1D(v_{>i}),
  \qquad
  x\text{ begins with }0,
\]
with the word weight of
\[
  y=v_{<i}x.
\]
We show that $(i,x)\mapsto y$ is a weight-preserving bijection onto
$\mathcal H(v)$.

Let $a=v_{<i}$ and $s=v_{>i}$, and write $x=x^-1$.  Since
$x^-\preceq s$, the word
$y^-=ax^-$ is a subsequence of $v$: use the literal prefix $a$, skip the
letter $v_i=1$, and embed $x^-$ in $s$.  On the other hand, $y\not\preceq v$.  Otherwise, since $x$ is nonempty and
begins with $0$, applying \cref{lem:embedding-position} to
$y=ax\preceq v=a1s$ would give $x\preceq s$, contrary to
$x\not\preceq s$.  Hence $y\in\partial_1D(v)$.  Moreover $y$ and $v$ agree
before position $i$, while $y_i=0<1=v_i$, so $y\lexlt v$.

Conversely, let $y\in\partial_1D(v)$ with $y\lexlt v$, and let $i$ be the
first position at which $y$ and $v$ differ.  The case in which $y$ is a
proper prefix of $v$ is impossible because then $y\preceq v$.  Thus
$v_i=1$ and $y_i=0$.  Write $y=v_{<i}x$.  Since $x$ begins with $0$ and ends with $1$, its proper prefix $x^-$ is
nonempty and begins with $0$.  From
\[
  y^-=v_{<i}x^-\preceq v=v_{<i}1v_{>i},
\]
\cref{lem:embedding-position} gives $x^-\preceq v_{>i}$ directly.  If
$x\preceq v_{>i}$, then $y\preceq v$, a contradiction.  Hence
$x\in\partial_1D(v_{>i})$, and $x$ begins with $0$.

The first differing position $i$ is unique, so the correspondence is a
bijection.  The first identity follows from \eqref{eq:H-source}.  The union
in \eqref{eq:L-language} is disjoint because every word in
$\mathcal H(v)$ lies outside $D(v)$, and \eqref{eq:J-enumerator} follows
from $J_v=I_v+H_v$.
\end{proof}

The lexicographic language model gives the following degree bound.

\begin{lemma}\label{lem:one-degree}
If $o=\oo(v)$, then every word in $\mathcal L(v)$ has at most $o$ ones.
Equivalently,
\[
  \deg_YJ_v\le o.
\]
\end{lemma}

\begin{proof}
The assertion is immediate for $D(v)$.  If $y\in\mathcal H(v)$, use the
unique representation in the proof of \cref{prop:H-lex}:
\[
  y=v_{<i}x,
  \qquad
  v_i=1,
  \qquad
  x=x^-1,
  \qquad
  x^-\preceq v_{>i}.
\]
Then
\[
 \oo(y)
 \le \oo(v_{<i})+\oo(v_{>i})+1
 =\oo(v).
\]
\end{proof}

\section{One-sided deletion and the half-plane bound}

The language $\mathcal L(v)$ is closed under one deletion of $1$.

\begin{lemma}\label{lem:lex-delete}
If a word $z$ is obtained from a word $y$ by deleting one occurrence of
$1$, then $z\lexlt y$.
\end{lemma}

\begin{proof}
Write $y=A1B$ and $z=AB$.  If $B=1^m0C$, then after the common prefix
$A1^m$, the word $z$ has $0$ while $y$ has $1$.  If $B=1^m$, then $z$ is a
proper prefix of $y$.
\end{proof}

\begin{proposition}\label{prop:delete-one}
If $y\in\mathcal L(v)$ and one occurrence of $1$ is deleted from $y$, the
resulting word still belongs to $\mathcal L(v)$.
\end{proposition}

\begin{proof}
For $y\in D(v)$ this follows from subsequence closure.  Now let
$y\in\mathcal H(v)$.  Thus $y=y^-1$, with $y^-\in D(v)$ and
$y\lexlt v$.

If the deleted letter is the final $1$, the result is $y^-\in D(v)$.  If a
nonfinal $1$ is deleted, call the resulting word $z=z^-1$.  It still ends in the
original final $1$, and $z^-$ is a subsequence of $y^-$, hence belongs to
$D(v)$.  Therefore either $z\in D(v)$, or
$z\in\partial_1D(v)$.  In the latter case, \cref{lem:lex-delete} gives
$z\lexlt y\lexlt v$, so $z\in\mathcal H(v)$.  In both cases
$z\in\mathcal L(v)$.
\end{proof}

Write
\begin{equation}\label{eq:j-coeff}
  J_v=\sum_{a,b\ge0}j_{a,b}X^aY^b.
\end{equation}

\begin{corollary}\label{cor:marked-deletion}
For $b\ge1$,
\begin{equation}\label{eq:marked-ineq}
  b\,j_{a,b}\le(a+b)j_{a,b-1}.
\end{equation}
In particular, for every $m\ge1$,
\begin{equation}\label{eq:diag-subdiag}
  j_{m,m}\le2j_{m,m-1}.
\end{equation}
\end{corollary}

\begin{proof}
Count pairs consisting of a word
$y\in\mathcal L(v)$ with $(\zz(y),\oo(y))=(a,b)$ and a marked occurrence
of $1$.  There are $b j_{a,b}$ such pairs.  Delete the marked $1$.  By
\cref{prop:delete-one}, the resulting word $z$ belongs to $\mathcal L(v)$
and has composition $(a,b-1)$.  Record also the one of the
$|z|+1=a+b$ insertion slots from which the marked $1$ was deleted.  The word
$z$ together with that marked slot uniquely recovers the source pair, so the
map is injective into a set of size $(a+b)j_{a,b-1}$.  Setting $a=b=m$ in
\eqref{eq:marked-ineq} gives \eqref{eq:diag-subdiag}.
\end{proof}

We now evaluate the factorization from \cref{cor:cross-factor}.  For a
single monomial, direct inspection of \eqref{eq:phi-def} gives
\begin{equation}\label{eq:monomial-phi}
 \PhiNeg\!\left(\Lambda X^aY^b\right)
 =\begin{cases}
   2^{-2m-1},&a=b=m,\\
   -2^{-2b-2},&a=b+1,\\
   0,&\text{otherwise}.
  \end{cases}
\end{equation}
Indeed, when $a-b\ge2$, the two positive child terms and the negative parent
term cancel after evaluation at $X=Y=1/2$; when $a-b<0$, none lies in the
strict half-plane $a>b$.

Since the empty word belongs to $\mathcal L(v)$, one has $j_{0,0}=1$.
Applying \eqref{eq:monomial-phi} to \eqref{eq:j-coeff} yields
\begin{align}
 \PhiNeg(\Lambda J_v)
 &=\sum_{m\ge0}2^{-2m-1}j_{m,m}
   -\sum_{m\ge0}2^{-2m-2}j_{m+1,m}
   \notag\\
 &=\frac12+
   \sum_{m\ge1}
   \left(2^{-2m-1}j_{m,m}
   -2^{-2m}j_{m,m-1}\right).
 \label{eq:halfplane-defects}
\end{align}
Every summand in parentheses is nonpositive by
\eqref{eq:diag-subdiag}; hence
\begin{equation}\label{eq:weak-bound}
  \PhiNeg(\Lambda J_v)\le\frac12.
\end{equation}

It remains to absorb the exceptional monomial $s_v$.  Put
\[
  z=\zz(v),
  \qquad
  o=\oo(v).
\]
If $z\le o$, then $s_v=X^{z+1}Y^{o+1}$ does not lie in the strict half-plane,
so \eqref{eq:weak-bound} and \eqref{eq:cross-factor} give
$\PhiNeg(C_v)\le1/2$.

Suppose now that $z>o$.  By \cref{lem:one-degree},
  $$j_{m,m}=j_{m+1,m}=0\ \text{for}\ m\geq o+1.$$
On the other hand, $D(v)$ contains a subsequence formed from all $o$ ones of
$v$ and any $o+1$ of its zeros, so
\[
  j_{o+1,o}\ge1.
\]
The $m=o+1$ term of \eqref{eq:halfplane-defects} is therefore at most
$-2^{-2o-2}$, while every other defect is nonpositive.  Thus
\begin{equation}\label{eq:strict-room}
  \PhiNeg(\Lambda J_v)\le\frac12-2^{-2o-2}.
\end{equation}
Also
\[
  \PhiNeg(s_v)=2^{-(z+o+2)}\le2^{-2o-3}.
\]
Combining this with \eqref{eq:strict-room}, and using
$\PhiNeg(1)=0$, gives
\[
  \PhiNeg(C_v)
  \le\frac12-2^{-2o-3}
  <\frac12.
\]
Together with the case $z\le o$, we have proved
\[
  \PhiNeg(C_v)\le\frac12
\]
for every word $v$.  The identity in
\cref{prop:first-exit} now gives
\[
  |S_{t,k}|\le2^{k-1},
\]
which proves \cref{thm:main}.

\section{Acknowledgements}
The authors thank ChatGPT 5.6 Pro for assistance with some of the mathematical
work presented in this paper. The authors subsequently verified the argument
and take full responsibility for the final content. An accompanying Lean
formalization is available at
\url{https://github.com/ifeelok92/tu-deng-conjecture-lean-formal-proof}. It
provides an end-to-end machine-checked proof of the Tu--Deng conjecture,
including the principal intermediate results and the final reduction to the
original modular counting statement.

\end{document}